\documentclass[letterpaper, 12pt]{article} 
\usepackage{amsmath,amsfonts,amssymb,amsthm,xypic,xy,hyperref,mathpazo}
\xyoption{all}

{}
\newtheorem{theorem}{Theorem}[section]
\newtheorem{corollary}[theorem]{Corollary}

\newtheorem{prop}[]{Proposition}
\newtheorem{defn}[]{Definition}
\newtheorem{example}[theorem]{Example}
\newtheorem{remark}[]{Remark}

\newcommand{\nc}{\newcommand}
\nc{\A}{\mathcal A}\nc{\FH}{\mathcal H} \nc{\CC}{\mathbb C}
\nc{\JJ}{\mathcal J} \nc{\KK}{\mathbb K} \nc{\RR}{\mathbb R}
\nc{\LL}{\mathcal L} \nc{\Ll}{\ell} \nc{\NN}{\mathbb N}
\nc{\ZZ}{\mathbb Z} \nc {\HH}{\mathbb H} \nc {\OO}{\mathcal{O}}
\nc{\lra}{\longrightarrow} \nc{\bdot}{\bullet} \nc{\w}{\omega}
\nc{\dd}{\mathcal{D}}

\nc{\id}{\mathrm{id}}
\newcommand{\End}{\mathrm{End}}
\newcommand{\Ann}{\mathrm{Ann}}
\newcommand{\delbar}{\overline{\partial}}
\newcommand{\so}{\mathfrak{so}}
\newcommand{\IP}[1]{\langle #1\rangle}
\nc{\ba}{\overline} \nc{\del}{\partial} \nc{\AAA}{\mathcal{A}}
\nc{\T}{\tilde}\nc{\Hom}{\mathrm{Hom}}\nc{\ch}{\mathrm{ch}}\nc{\Td}{\mathrm{Td}}
\usepackage{color}
\definecolor{tocolor}{rgb}{.1,.1,.5}
\definecolor{urlcolor}{rgb}{.2,.2,.6}
\definecolor{linkcolor}{rgb}{.1,.1,.6}
\definecolor{citecolor}{rgb}{.6,.2,.1}
\hypersetup{backref=true, colorlinks=true, urlcolor=urlcolor, linkcolor=linkcolor, citecolor=citecolor}

\begin{document}
\title{\bf Branes on Poisson varieties}
\date{}
\author{ Marco Gualtieri}
\maketitle

\centerline{\it Dedicated to Nigel Hitchin on the occasion of his sixtieth birthday. }

\abstract{We first extend the notion of connection in the context of
Courant algebroids to obtain a new characterization of generalized
K\"ahler geometry.  We then establish a new notion of isomorphism
between holomorphic Poisson manifolds, which is non-holomorphic in
nature. Finally we show an equivalence between certain
configurations of branes on Poisson varieties and generalized
K\"ahler structures, and use this to construct explicitly new
families of generalized K\"ahler structures on compact holomorphic
Poisson manifolds equipped with positive Poisson line bundles (e.g.
Fano manifolds).  We end with some speculations concerning the
connection to non-commutative algebraic geometry.}

\section{Introduction}
In this paper we shall take a second look at a classical structure
in differential and algebraic geometry, that of a holomorphic
Poisson structure, which is a complex manifold with a holomorphic
Poisson bracket on its sheaf of regular functions.  The structure is
determined, on a real smooth manifold $M$, by the choice of a pair
$(I,\sigma_I)$, where $I$ is an integrable complex structure tensor
and $\sigma_I$ is a holomorphic Poisson tensor.  We shall view
$(I,\sigma_I)$ not as we normally do but instead as a
\emph{generalized complex structure}, in the sense of
\cite{MR2013140}. In so doing, we shall obtain a new notion of
equivalence between the pairs $(I,\sigma_I)$ which does not imply
the holomorphic equivalence of the underlying complex structures.

In studying this equivalence relation we are naturally led to an
unexpected connection to \emph{generalized K\"ahler geometry}, as
defined in~\cite{Gualtieri:rp}, and to a method for constructing certain
examples of these structures which extends the recent work of
Hitchin constructing bi-Hermitian metrics on Del Pezzo
surfaces~\cite{MR2371181}; in particular we obtain similar families of
bi-Hermitian metrics on all smooth Poisson Fano varieties, and in
fact on any smooth Poisson variety admitting a positive Poisson line
bundle.  We therefore give an explicit construction of a subclass of
the generalized K\"ahler structures proven to exist by the
generalized K\"ahler stability theorem of Goto~\cite{Goto:2007if}.

In both these efforts we shall find it useful to introduce an
extension of the notion of connection on a vector bundle, to allow
differentiation not only in the tangent but also the cotangent
directions; we call such a structure a \emph{generalized
connection}.  We also show that in the presence of a generalized
metric, there is a canonical connection $D$ which plays the role of
the Levi-Civita connection in K\"ahler geometry: namely, we show
that $(\JJ,G)$ is generalized K\"ahler if and only if $D\JJ=0$.

In the final section we make some speculative comments concerning
the relationship between generalized K\"ahler geometry and
non-commutative geometry, a topic we hope to clarify in the future.

I would like to thank Nigel Hitchin for many illuminating
discussions, and also the organizers of his birthday conference,
especially Luis \'Alvarez-C\'onsul and Oscar Garc\'ia-Prada, for a
stimulating event. I would also like to thank Mike Artin, Gil
Cavalcanti, Izzet Coskun and Pavel Etingof for helpful discussions.  Finally, I thank 
Yicao Wang for pointing out an incorrect definition of the generalized Bismut connection in a previous draft and for correcting it. 
\section{Gerbe trivializations}

Let $M$ be a manifold equipped with a $U(1)$ gerbe with connection
(specifically, a gerbe with connective structure in the sense
of~\cite{MR1197353}). This determines canonically a \emph{Courant
algebroid} $E$ over $M$, in the same way that a $U(1)$ principal
bundle $P$ determines an Atiyah Lie algebroid $E = TP/U(1)$ over
$M$. See~\cite{sevwein,MR2253158} for details of this construction,
and see~\cite{MR998124,MR1958835,MR1472888} for details concerning Courant
algebroids; we review their main properties presently.

The Courant algebroid $E$ is an extension of real vector bundles
\begin{equation}\label{ca}
\xymatrix{0\ar[r]& T^*\ar[r]^{\pi^*}& E\ar[r]^\pi& T\ar[r]&0},
\end{equation}
where $T$ and $T^*$ denote the tangent and cotangent bundles of $M$.
Further, $E$ is equipped with a nondegenerate symmetric bilinear
form $\IP{\cdot,\cdot}$ of split signature, such that
$\IP{\pi^*\xi,a} = \xi(\pi(a))$.  Finally, there is a bilinear
\emph{Courant bracket} $[\cdot,\cdot]$ on $C^\infty(E)$ such that
\begin{itemize}
\item $[[a,b],c] = [a,[b,c]] - [b,[a,c]]$  (Jacobi identity),
\item $[a,fb] = f[a,b] + (\pi(a) f)b$ (Leibniz rule),
\item $\pi(a)\IP{b,c} = \IP{[a,b],c} + \IP{b,[a,c]}$ (Invariance of
bilinear form),
\item $[a,a] = \pi^* d\IP{a,a}$ (Skew-symmetry anomaly).
\end{itemize}

The choice of an isotropic complement to $T^*$ in $E$ is a
contractible one, and so an isotropic splitting $s:T\lra E$ of the
sequence~\eqref{ca} always exists.  Each such splitting determines a
closed 3-form $H\in \Omega^3(M)$, given by
\begin{equation}\label{ache}
(i_Yi_XH)(Z) = \IP{[s(X),s(Y)],s(Z)}.
\end{equation}
The cohomology class $[H]/{2\pi}\in H^3(M,\RR)$ is independent of
the choice of splitting, and  coincides with the image of the
Dixmier-Douady class of the gerbe in real cohomology.  Furthermore,
$[H]$ classifies the Courant algebroid up to isomorphism, as shown
by~\cite{sevwein}.

Courant algebroids may be naturally pulled back by the inclusion
$S\subset M$ of a submanifold; as a bundle over $S$, the result is
simply given by
\begin{equation}
E_S = \frac{\pi^{-1}(TS)}{\Ann(TS)},
\end{equation}
and its bracket and inner product are inherited in a straightforward
manner.

A trivialization of the gerbe along $S$ induces a Courant
trivialization in the following sense:
\begin{defn}
A \emph{Courant trivialization} along $S$ is an integrable isotropic
splitting $s: TS\lra E_S$ of the pullback Courant algebroid.  Such
trivializations exist if and only if $\iota^*[H]=0$ for
$\iota:S\hookrightarrow M$ the inclusion.
\end{defn}
Integrability is the requirement that the subbundle $s(TS)\subset
E_S$ be closed under the Courant bracket. Integrable maximal
isotropic subbundles of a Courant algebroid are called Dirac
structures; therefore $s(TS)$ is simply a Dirac structure transverse
to $T^*S$.  As a result of a Courant trivialization, $E_S$ is
canonically isomorphic to $TS\oplus T^*S$ with its natural pairing
and the bracket
\[
[X+\xi,Y+\eta] = L_X(Y+\eta) - i_Yd\xi.
\]
Now suppose that $S_0, S_1\subset M$ are submanifolds with smooth
intersection, and suppose we have gerbe trivializations on each of
them.  Then on $X=S_0\cap S_1$ we obtain a pair of gerbe
trivializations, which must differ by a line bundle $L_{01}$ with
$U(1)$ connection $\nabla_{01}$.  Let $s_0, s_1$ be the splittings
of $E_X$ determined by the two gerbe trivializations. Then
$s_1-s_0:TX\lra T^*X$ is given by $X\mapsto i_X F_{01}$ where
$F_{01}\in \Omega^2(M)$ is the curvature of $\nabla_{01}$.

The notion of Courant trivialization provides a convenient way of
characterizing isomorphisms of Courant algebroids, as in the
following example.  The notation $\overline{E}$ denotes the Courant
algebroid $E$, equipped with the opposite bilinear form
$-\IP{\cdot,\cdot}$.
\begin{example}\label{couriso}
Let $E_M, E_N$ be Courant algebroids over the manifolds $M,N$
respectively.  They are isomorphic precisely if there is a Courant
trivialization of the product Courant algebroid
$\overline{E}_M\times E_N$ along the graph of a diffeomorphism
$\varphi:M\lra N$ in the product $M\times N$.
\end{example}
\section{Generalized connections}

Let $E$ be a Courant algebroid as in the previous section.  In
keeping with the notion that the Courant algebroid is an analogue of
the tangent bundle, we have the following generalization of the
usual notion of connection.
\begin{defn}
A \emph{generalized connection} on a vector bundle $V$ is a
first-order linear differential operator
\[
D: C^\infty(V)\lra C^\infty(E\otimes V)
\]
such that $D(fs)=fDs +(\pi^*df)\otimes s$. Furthermore, if $V$ has a
Hermitian metric $h$, then $D$ is \emph{unitary} when
\[
d (h(s,t)) = h(Ds,t) + h(s,Dt).
\]
\end{defn}
If $s:T\lra E$ is any splitting (not necessarily isotropic) of the
Courant algebroid, then using the decomposition $E = s(T)\oplus T^*
\cong T\oplus T^*$, we have
\begin{equation}\label{decomp}
D = \nabla + \chi,
\end{equation}
where $\nabla$ is a usual unitary connection and  $\chi$ is a vector
field with values in the bundle of skew-adjoint endomorphisms of
$V$, i.e. $\chi\in C^\infty(T\otimes \mathfrak{u} ( V))$.  The
tensor $\chi$ is independent of the choice of splitting, and we note
that if $V$ is of rank 1, $\chi$ is simply a vector field on the
manifold.

With respect to a different splitting $s'$, such that
\[
s'-s = \theta: T\lra T^*,
\]
we obtain a different decomposition $D = \nabla' + \chi$, where
$\nabla' = \nabla + \theta(\chi)$.

A generalized connection has a natural curvature operator: for
$a,b\in C^\infty(E)$, we define
\[
R(a,b) = [D_a,D_b] - D_{[a,b]} \in C^\infty(\mathfrak{u}(V)).
\]
This becomes tensorial in $a,b$ when restricted to a Dirac structure
$L\subset E$:
\[
R|_L \in C^\infty(\wedge^2 L^*\otimes\mathfrak{u}(V)).
\]
If $L=T^*$, for example, we obtain a bivector with values in the
skew-adjoint endomorphisms,  $R|_{T^*} = [\chi,\chi]$.

The tensorial curvatures $R_{s'},R_{s}$ associated to integrable
splittings $s,s'$ of $E$ with $s'-s = F\in\Omega^2_{cl}(M)$ may be
compared by projection to $T$:
\[
R_{s'}= R_s + d^\nabla(a) + a\wedge a,
\]
where $a = F(\chi)$. Therefore if $V$ has rank 1, we have that $\chi
= iX$ for a real vector field $X$, and
\[
R_{s'} - R_s = d i_X F.
\]

In the particular case that we have a generalized connection $D$ on
$E$ itself, it is natural to compare the connection derivative with
the Courant bracket; we therefore introduce the \emph{torsion} of
$D$, and leave it as an exercise to verify it is well-defined.
\begin{defn}
The torsion $T_D\in C^\infty(\wedge^2 E\otimes E)$ of a generalized
connection $D$ on $E$ itself is defined by
\[
T(a,b,c)=\IP{D_ab-D_ba-[a,b]_{sk},c} +
\tfrac{1}{2}(\IP{D_ca,b}-\IP{D_cb,a}),
\]
where $[a,b]_{sk}=\tfrac{1}{2}([a,b]-[b,a])$. If $D$ preserves the
canonical bilinear form $\IP{\cdot,\cdot}$ on $E$, then $T_D$ is
totally skew, i.e. $T_D\in C^\infty(\wedge^3 E)$.
\end{defn}

A generalized Riemannian metric on the Courant algebroid $E$ is the
choice of a maximal positive-definite subbundle $C_+\subset E$; this
reduces the $O(n,n)$ structure of $E$ to $O(n)\times O(n)$, and
defines a positive-definite metric on $E$:
\[
 G(\cdot,\cdot) = \IP{\cdot,\cdot}|_{C_+} - \IP{\cdot,\cdot}|_{C_-},
\]
where $C_- = C_+^\bot$ is the orthogonal complement with respect to
$\IP{\cdot,\cdot}$.
We now describe a construction of a canonical connection associated
to the choice of such a metric, inspired by calculations
in~\cite{Gualtieri:rp,MR2253158} relating metric connections with skew torsion
to the Courant bracket.

The $G$-orthogonal complement to $T^*$ is an isotropic splitting
$C_0\subset E$ and we identify it with $T$, so that $G$ induces a
splitting $E=T\oplus T^*$. The Courant bracket in this splitting is
\[
 [X+\xi,Y+\eta]_H=[X,Y] + L_X\eta - i_Yd\xi + i_Yi_X H,
\]
where $H\in\Omega^3_{cl}(M)$ is defined by~\eqref{ache}.  The splitting also defines an anti-orthogonal automorphism $C:E\lra E$ defined by $C(X+\xi)=X-\xi$, which satisfies $C(C_\pm) = C_\mp$.  It also has the property, for $Z,W\in C^\infty(E)$:
\[
 [CZ,CW]_H = C([Z,W]_{-H}).
\]
\begin{theorem}\label{lclc}
Let $C_+\subset E$ be a maximal positive-definite subbundle, i.e. a
generalized metric, as above.  Let $C:E\lra E$ be the map defined above.  Write $Z=Z_++Z_-$ for the orthogonal projections of $Z\in C^\infty(E)$ to $C_\pm$. Then the operator
\begin{equation}\label{lc}
D_Z(W) = [Z_-, W_+]_+ + [Z_+, W_-]_- +[CZ_-, W_-]_- + [CZ_+,W_+ ]_+
\end{equation}
defines a generalized connection on $E$, preserving both
$\IP{\cdot,\cdot}$ and the positive-definite metric $G$.   We call this the generalized Bismut connection.
\end{theorem}
\begin{proof}
Using the properties of the Courant bracket and the orthogonality
$C_+=C_-^\bot$, we have immediately the property $D_{fZ}W=fD_ZW$,
for $f\in C^\infty(M)$.  We also have
\begin{align*}
D_Z(fW) &= fD_Z(W) + (Z_-f)W_+ + (Z_+f) W_- + (Z_- f)W_- + (Z_+f)W_+\\
&=fD_Z(v) + (Zf) W,
\end{align*}
proving that $D$ is a generalized connection.

It is clear from~\eqref{lc} that $C_\pm$ are preserved by the connection, since $D_ZW$ has nonzero component in $C_\pm$ if and only if $W$ does. Hence we obtain a decomposition
\begin{equation}\label{pmdecomp}
D = D^+\oplus D^-,
\end{equation}
where $D^\pm$ are generalized connections on $C_\pm$ respectively.

To prove that $D$ preserves the canonical metric $\IP{\cdot,\cdot}$ as well as the metric $G$, we show that $D^\pm$ preserve the induced metrics on $C_\pm$.  Let $V,W\in C^\infty(C_+)$, and $Z\in C^\infty(E)$. Then
\begin{align*}
 Z_+\IP{V,W} &= (CZ_+)\IP{V,W}\\
&=\IP{[CZ_+,V],W]} +\IP{V,[CZ_+,W]}\\
&= \IP{D_{Z_+}V,W} + \IP{V,D_{Z_+}W}.
\end{align*}
Also, we have
\begin{align*}
 Z_-\IP{V,W} &= \IP{[Z_-,V],W} + \IP{V,[Z_-,W]}\\
&=\IP{D_{Z_-}V,W} + \IP{V,D_{Z_-}W}.
\end{align*}
Summing these two results, we see that $D^+$ preserves the metric on $C_+$; the same argument holds for $C_-$, completing the proof.
\end{proof}
The generalized connections $D^\pm$ in the decomposition~\eqref{pmdecomp} define tensors $\chi^\pm \in C^\infty(T\otimes \so(C_\pm))$, via the decomposition~\eqref{decomp}.  We see now that these vanish, since for $Z\in C^\infty(T^*)$ and $W\in C^\infty(C_\pm)$, we have
\begin{align*}
\chi^\pm_ZW &= D_ZW = [Z_\mp, W_\pm]_\pm + [CZ_\pm, W_\pm]_\pm\\
&= [Z_\mp + (CZ)_\mp, W_\pm]_\pm\\
&=0,
\end{align*}
where we use the fact that $Z\in T^*$ if and only if $CZ=-Z$. 

As a result of this, we conclude that $D^\pm$ may be viewed as usual metric connections $\nabla^\pm$ on $T$, via the projection isomorphisms $\pi_\pm:C_\pm\lra T$, i.e.
\[
D^\pm = \pi_\pm^{-1} \nabla^\pm \pi_\pm. 
\]
The connections $\nabla^\pm$ may be described as follows:
\begin{align*}
\nabla^\pm_X Y &= 2\pi_\pm D_X^\pm Y_\pm\\
&=4\pi_\pm D^\pm_{X_\mp} Y_\pm\\
&= 4\pi_\pm [X_\mp,Y_\pm]_\pm.
\end{align*}

We may easily compute the torsion $T^\pm$ of the connections $\nabla^\pm$, 
for vector fields $X,Y,Z$:
\begin{align*}
2g(T^+(X,Y),Z) &= \IP{T^+(X,Y), Z_+}\\
&=\IP{ 4\pi_+[X_-,Y_+]_+ - 4\pi_+[Y_-,X_+]_+ - \pi[X,Y]  , Z_+}\\
&=\IP{2[X_-,Y_+]_+ + 2[X_+,Y_-]_+ -[X,Y] + i_Yi_X H, Z_+}\\
&= 2H(X,Y,Z) - 2\IP{[X_+-X_-,Y_+-Y_-],Z_+}\\
&=2H(X,Y,Z),
\end{align*}
by the fact that $[X_+-X_-,Y_+-Y_-]=0$ since the Courant bracket vanishes on 1-forms.  A similar calculation gives $g(T^-(X,Y),Z) = -H(X,Y,Z)$. 

The above calculation shows that $\nabla^\pm$ coincide with the Bismut connections with totally skew torsion $\pm H$.  In this way, we have essentially repeated the observation  of~\cite{MR2253158} that
$\nabla^\pm$ may be conveniently expressed in terms of the Courant
bracket.  To summarize, the generalized Bismut connection is essentially a usual connection on $E$ which restricts to $C_\pm$ to give the Bismut connections with torsion $\pm H$.

\begin{prop}\label{torscalc}
The torsion $T_D$ of the generalized Bismut connection lies in $\wedge^3 C_+\oplus \wedge^3 C_-\subset \wedge^3 E$, and is given by 
\[
T_D = \pi_+^* H + \pi_-^* H.
\]
\end{prop}
\begin{proof}
First we show that $T(C_+,C_-,\cdot)=0$, so that $T\in
C^\infty(\wedge^3 C_+\oplus \wedge^3 C_-)$.  Let $x\in
C^\infty(C_+)$, $y\in C^\infty(C_-)$ and $z\in C^\infty(E)$. Then
\begin{align*}
 T_D(x,y,z) &= \IP{D_xy-D_yx -[x,y],z}\\
 &=\IP{[x,y]_- -[y,x]_+ - [x,y],z}=0,
\end{align*}
as required. 

Now take $x,y,z\in C^\infty(C_+)$.  Since $\chi_D=0$, we have the identity
\begin{align*}
\IP{D_z x, y} -\IP{D_z y,x} &=  \IP{D_{Cz}x,y}-\IP{D_{Cz}y,x}\\
&=\IP{[Cz,x],y}-\IP{[Cz,y],x}\\
&=\IP{[x,y]-[y,x],Cz}.
\end{align*}
Therefore the torsion is given by
\begin{align*}
T(x,y,z)&=\IP{D_xy-D_yx-[x,y]_{sk},z} + \tfrac{1}{2}\IP{[x,y]-[y,x],Cz}\\
&=\IP{D_xy-D_yx,z} +\IP{[x,y], Cz-z}\\
&= g(\nabla^+_XY-\nabla^+_YX-[X,Y],Z)\\
&=H(X,Y,Z).
\end{align*}
A similar calculation for $x,y,z\in C^\infty(C_-)$ gives $T(x,y,z)=H(x,y,z)$ as well, yielding the result.
\end{proof}

As we have explained, the generalized Bismut connection $D$ is completely determined by a usual connection on $T\oplus T^*$.  Using the decomposition~\eqref{pmdecomp}, and the fact that the Bismut connections satisfy $\nabla^\pm = \nabla \pm \tfrac{1}{2}g^{-1}H$ for $\nabla$ the Levi-Civita connection, we may write $D$ explicitly with respect to the splitting $E = T\oplus T^*$, and for $X\in C^\infty(T)$,  as follows:
\[
D_X = \begin{pmatrix}\nabla_X & \tfrac{1}{2}\wedge^2 g^{-1} (i_X H)\\\tfrac{1}{2}i_XH & \nabla^*_X\end{pmatrix}
\]
The significance of this connection in the context of generalized geometry was first understood and investigated by Ellwood in~\cite{Ellwood:2006ya}.  Here we simply view it as a generalized connection\footnote{The author thanks Yicao Wang for pointing this out and correcting an error in the previous version.} mainly for the purpose of highlighting its properties and defining its torsion tensor.    

\section{Generalized holomorphic bundles and branes}

Suppose now that we have a generalized complex structure $\JJ$ on
$(M,E)$, which is an orthogonal almost complex structure $\JJ:E\lra
E$ whose $+i$ eigenbundle $L\subset E\otimes\CC$ is closed under the
Courant bracket~\cite{MR2013140}.  We now describe how the structures
in the previous two sections may be made \emph{compatible} with
$\JJ$.

\subsection{Generalized holomorphic bundles}
The integrability of $\JJ$ guarantees that $L=\ker(\JJ-i1)$ is a
complex Lie algebroid, with associated de Rham complex
\begin{equation}
\xymatrix{C^\infty(\wedge^k
L^*)\ar[r]^{d_L}&C^\infty(\wedge^{k+1}L^*)}
\end{equation}
A complex vector bundle equipped with a flat $L$-connection is
called a generalized holomorphic bundle~\cite{Gualtieri:qr}.  Therefore,
generalized holomorphic bundles form a category of Lie algebroid
representations in the sense of~\cite{MR1726784}.

In the case that $\JJ$ is a usual complex structure, for instance, a
generalized holomorphic bundle consists of a holomorphic bundle $V$,
together with a holomorphic section $\Phi\in H^0(M, T_{1,0}M\otimes
\End(V))$ satisfying
\[
\Phi\wedge\Phi = 0 \in H^0(M, \wedge^2T_{1,0}M\otimes\End(V).
\]
Note that if $M$ is holomorphic symplectic, then $T_{1,0}$ is
isomorphic to $T^*_{1,0}$, and $\Phi$ may be viewed as a Higgs
bundle, in the sense of~\cite{Simpson}.

In the case that $\JJ$ is a symplectic structure, a generalized
holomorphic bundle is simply a flat bundle.

\begin{defn}
A unitary generalized connection $D$ on a complex vector bundle $V$
is compatible with $\JJ$ when its curvature along $L = \ker
(\JJ-i1)$ is zero.
\end{defn}
It follows immediately that the restriction of $D$ to $L$ defines a
flat $L$-module structure on $V$, making $V$ a generalized
holomorphic bundle. Conversely, suppose $V$ is a $\JJ$-holomorphic
bundle, i.e. it is equipped with an $L$-connection as follows:
\begin{equation}\label{lmod}
\delbar: C^\infty(V)\lra C^\infty(L^*\otimes V),\ \ \ \delbar^2 = 0.
\end{equation}
This operator has symbol sequence given by wedging with $\sigma(\xi)
= \tfrac{1}{2}(1+i\JJ)\xi\in \overline L$, where we identify
$L^*=\overline L$ using the metric on $E$.

Choosing a Hermitian metric $h$ on the bundle $V$, so that
$\overline V\simeq V^*$, we may view the complex conjugate
of~\eqref{lmod},
\[
\del: C^\infty(\overline V)\lra C^\infty(L\otimes \overline V),
\]
as a $L$-connection on $V^*$; we then form the dual $\del^*$ of this
partial connection. Finally we form the sum
\[
D = \delbar + \del^*: C^\infty(V)\lra C^\infty((\overline L\oplus
L)\otimes V) = C^\infty(E\otimes V),
\]
which has symbol $\sigma + \bar\sigma = \pi^*$.  Hence it defines a
generalized connection on $V$. We summarize the above in the
following
\begin{prop}\label{herm}
Let $V$ be a complex vector bundle with $\JJ$-holomorphic structure
given by $\delbar$, and choose a Hermitian metric on $V$.  Then the
operator
\[
D = \delbar + \del^* : C^\infty(V)\lra C^\infty(E\otimes V)
\]
is the unique unitary generalized connection extending $\delbar$.
\end{prop}

When $V$ is a line bundle, there is a useful formula for the
generalized connection 1-form in terms of a holomorphic
trivialization, analogous to the Poincar\'{e}-Lelong formula for the
Chern connection on a Hermitian holomorphic line bundle.
\begin{prop}[Generalized Poincar\'{e}-Lelong formula]\label{PL}
Let $V$ be a generalized holomorphic hermitian line bundle, and let
$s\in C^\infty(V)$ be a holomorphic section.  Where it is nonzero,
it defines a trivialization of the unitary generalized connection,
$D = d + i\A$, where
\[
\A = \JJ d\log |s| \in C^\infty(E).
\]
\end{prop}
\begin{proof}
Whenever $s$ is nonzero, we have
\[
D\frac{s}{|s|} = i\A \frac{s}{|s|}.
\]
Taking the projection to $\overline L$, we obtain
\[
d_L \log |s| = i \A^{0,1},
\]
so that $\A = -i(d_L-d_{\overline L})\log |s| = \JJ d\log|s|$, as
required.
\end{proof}
In particular, if $s$ is nonzero on an open dense set, then the
vector field $\pi\JJ d\log|s| = X$ must extend to a smooth vector
field on the whole of $M$, since $\pi(i\A)$ coincides with $\chi\in
C^\infty(T\otimes\mathfrak{u}(V))$, which is globally defined for
any generalized connection.  But the map $\pi\JJ|_{T^*}:T^*\lra T$
is actually a Poisson structure $Q\in C^\infty(\wedge^2 T)$
(see~\cite{Gualtieri:qr} for a discussion of this fact), and hence $s$
vanishes only along the zero locus of the Poisson structure $Q$,
which is a strong constraint on any generalized holomorphic section.

The above proposition may be used, by invoking the local existence
of nonvanishing holomorphic sections near points for which $\JJ$ is
regular (i.e. $Q$ has locally constant rank), to show that the
vector field $\chi$ of any Hermitian $\JJ$-holomorphic line bundle
must be a Poisson vector field.  It therefore defines a
characteristic class in the Poisson cohomology
of~\cite{Lichnerowicz}, which is the cohomology of the complex
$(C^\infty(\wedge^\bullet T), d_Q)$, where $d_Q\Pi = [Q,\Pi]$ is the
Schouten bracket with $Q$.
\begin{corollary}
The real vector field $X=-i\chi$ associated to any Hermitian
generalized holomorphic line bundle preserves the Poisson structure
$Q$, i.e. it is a Poisson vector field. Furthermore its Poisson
cohomology class $[X]\in H^1_{Q}(M)$ is independent of the Hermitian
metric.
\end{corollary}
\begin{proof}$X$ is Poisson since, by the proposition, it is locally
Hamiltonian on an open dense subset of $M$.  Hence $L_XQ=0$
everywhere.  Rescaling the Hermitian metric by a positive function
$e^f$, we obtain a new vector field $X'=X + \tfrac{1}{2}Qdf$, which
differs from $X$ by a global Hamiltonian vector field. Hence $[X]\in
H^1_{Q}(M)$ is independent of the choice of Hermitian structure.
\end{proof}

We may also deduce this result from the more general fact that the
tensor product of a $L$-module with a $\overline{L}$ module is a
Poisson module for $Q$ (This is a direct consequence of the fact
that the tensor product of the Dirac structures $L, \overline{L}$ is
the Dirac structure associated to $Q$, shown in~\cite{Gualtieri:qr}). For any
generalized holomorphic line bundle $V$, therefore, the trivial
bundle $V\otimes \overline{V}$ acquires a $Q$-module structure, and
therefore, as described in~\cite{MR1726784}, a characteristic class in
$H^1_Q(M)$.

There are always two natural $\JJ$-holomorphic line bundles on any
generalized complex manifold: the trivial bundle, for which $\chi=0$
(for the standard Hermitian structure), and the canonical line
bundle $K_\JJ$ of pure spinors associated to the maximal isotropic
subbundle $L\subset E\otimes\CC$.  Since
$K_\JJ\otimes\overline{K_\JJ}$ is naturally the determinant line
$\det T^*$, it follows that $[X] = [-i\chi]$ is actually the
\emph{modular class} of the Poisson structure $Q$, in the sense
of~\cite{MR1484598}.

\subsection{Generalized complex branes}
Suppose we have a submanifold $\iota:S\hookrightarrow M$ equipped
with a Courant trivialization $s:TS\lra E_S$.  The Dirac structure
$s(TS)\subset E_S$ may be canonically lifted to a maximal isotropic
subbundle of $\iota^*E$; this operation is called the push-forward
of Dirac structures~\cite{MR1973074}:
\[
\tau_S:=\iota_*s(TS) = \{e\in E\ : \ \pi(e)\in TS\ \text{and}\
e+\Ann(TS)\in s(TS)\}.
\]
Note that $\tau_S$ is an extension of the tangent bundle of $S$ by
its conormal bundle:
\[
\xymatrix{0\ar[r]&N^*S\ar[r]&\tau_S\ar[r]&TS\ar[r]&0}
\]
In the presence of the generalized complex structure, there is a
natural compatibility condition, as follows.
\begin{defn}
A \emph{generalized complex submanifold} is a trivialization of the
Courant algebroid along a submanifold $\iota:S\hookrightarrow M$
which is compatible with the generalized complex structure $\JJ$, in
the sense that
\begin{equation}
\JJ\tau_S = \tau_S,
\end{equation}
where $\tau_S = \iota_*s(TS)\subset \iota^*E$.
\end{defn}

As shown in~\cite{Gualtieri:qr}, in the complex case (and for the trivial
gerbe), generalized complex submanifolds correspond to holomorphic
submanifolds equipped with unitary holomorphic line bundles, whereas
in the symplectic case they correspond to Lagrangian submanifolds
equipped with flat line bundles \emph{or} the co-isotropic A-branes
of~\cite{MR2006226}.   A useful general example of a generalized
complex submanifold is the graph of an isomorphism of generalized
complex manifolds, as follows.  The notation $\overline\JJ$ denotes
the same endomorphism as $\JJ$ but in the opposite Courant algebroid
$\overline E$.
\begin{example}\label{isogc}
Let $(M,\JJ_M)$, $(N,\JJ_N)$ be generalized complex manifolds.  They
are isomorphic when there is a Courant algebroid isomorphism in the
sense of Example~\ref{couriso} which is a generalized complex
submanifold of the product $(M\times N, \overline \JJ_M \times
\JJ_N)$.
\end{example}

In~\cite{Gualtieri:qr}, it is shown that in the eigenspace decomposition with
respect to $\JJ$
\[
\tau_S\otimes\CC = \ell + \bar\ell,
\]
the $+i$ eigenbundle $\ell$ inherits a Lie bracket, by extending
sections randomly to sections over $M$ which remain $+i$
eigensections of $\JJ$, taking their Courant bracket and restricting
to $S$.  Thus $\ell$ becomes an elliptic complex Lie algebroid over
$S$. Therefore there is a notion of flat $\ell$-module.  These
$\ell$-modules are called branes in analogy to the physics
literature.
\begin{defn}
A \emph{generalized complex brane} is a vector bundle with flat
$\ell$-connection, supported over a generalized complex submanifold.
\end{defn}
\begin{remark}
Just as for generalized holomorphic bundles, we may choose to
represent branes using unitary connections with values in
$\tau_S^*$, i.e. operators
\[
D:C^\infty(V)\lra C^\infty(\tau_S^*\otimes V)
\]
with symbol given by the inclusion $T^*S\subset \tau_S^*$,  and with
vanishing curvature along $\ell\subset \tau_S\otimes\CC$.
\end{remark}
For a usual complex structure, a brane consists of a holomorphic
bundle $V$ supported on a complex submanifold $S\subset M$ together
with a choice of holomorphic section $\phi\in H^0(S,
N_{1,0}S\otimes\End(V))$ satisfying
\[
\phi\wedge\phi = 0\in H^0(S,\wedge^2 N_{1,0}S\otimes\End(V)),
\]
where $N_{1,0}S$ denotes the holomorphic normal bundle of $S$.

On the other hand, for a symplectic structure, branes are complex
flat bundles if they are supported on Lagrangian submanifolds; they
may also be supported on coisotropic submanifolds with holomorphic
structure transverse to the characteristic
foliation~\cite{MR2006226,Gualtieri:qr}, in which case they are transversally
holomorphic bundles, flat along the leaves.
\begin{example}[Higgs bundles] Let $C\subset X$ be a curve in
a complex surface equipped with holomorphic symplectic form
$\Omega\in H^0(X,\Omega^{2,0})$, e.g. a K3 surface.  Also, let
$V\lra C$ be a Higgs bundle in the sense of~\cite{HitHiggs}, i.e. a
holomorphic bundle together with a Higgs field $\theta\in
H^0(C,\Omega^1\otimes\End(V))$.  Since $C$ is Lagrangian with
respect to $\Omega$, we have an isomorphism $T_{1,0}C\lra
N^*_{1,0}C$, so that we may form $\phi = \Omega^{-1}\theta\in
H^0(C,N_{1,0}\otimes\End(V))$, making $(V,\phi)$ into a brane for
the complex structure.

On the other hand, if $(V,\theta)$ is a stable Higgs bundle, then by
the existence theorem~\cite{HitHiggs} for solutions to Hitchin's
equations we obtain a complex flat connection $\nabla$ on $V$,
rendering $(C,V,\nabla)$ into a symplectic brane with respect to
either the real or imaginary parts of $\Omega$.
\end{example}

\begin{example}\label{rest}
Let $V$ be a generalized holomorphic bundle, i.e. a complex vector
bundle equipped with a flat $L$-connection, where $L=\ker(\JJ-i1)$
for $\JJ$ a generalized complex structure.  Then the pullback of $V$
to any generalized complex submanifold $S\subset M$ defines a
generalized complex brane, as can be seen easily from the inclusion
$\ell\subset L$.
\end{example}

Another simple example of a generalized complex brane occurs when it
is supported on an isomorphism of generalized complex manifolds, as
in Example~\ref{isogc}.

\begin{prop}\label{transbranes}
Let $S\subset M_0\times M_1$ define an isomorphism of the
generalized complex manifolds $(M_0,\JJ_0)$, $(M_1,\JJ_1)$.  Then
the Lie algebroid $\ell$ is isomorphic to both $L_i=\ker(\JJ_i-i1)$,
so that branes on $S$ may be identified with generalized holomorphic
bundles on either manifold.
\end{prop}
\begin{proof}
Let $\pi_i:M_0\times M_1\lra M_i$ be the usual projection maps. The
subbundle $\ell\subset (\pi_0^*L_0\oplus \pi_1^*L_1)|_S$ is
transverse to both $\pi_i^*L_i$, as we now show.  If $x\in \ell\cap
\pi_0^*L_0$, for instance, then $(\pi_1)_*x =
\varphi_*((\pi_0)_*x)$, where $\varphi$ is the diffeomorphism
defining $S$. This implies $(\pi_i)_*x=0$ and $x\in N^*S\otimes\CC$,
which clearly is transverse to $\pi_0^*L_0$. Hence $x=0$, and
similarly for $L_1$.

This transversality means that we have isomorphic bundle maps onto
each factor:
\begin{equation}
\xymatrix{L_0 \ar[d]&\ell\ar[l]_{p_0}\ar[d]\ar[r]^{p_1}&L_1\ar[d]
\\M_0&S\ar[l]_{\pi_0}\ar[r]^{\pi_1}&M_1}
\end{equation}
We now show that the projections $p_0,p_1$ are isomorphisms of Lie
algebroids.

Given $Z\in C^\infty(S,\ell)$, let $X=p_0(Z)$ and $Y=p_1(Z)$.  Then
 $Z$ may be expressed as $Z = (\pi_0^*X+ \pi_1^*Y)|_S$. Computing the
bracket of $Z,Z'$, we may use the given extensions to $M_0\times
M_1$ and compute their Courant bracket:
\begin{equation}\begin{split}
[Z,Z'] &= [\pi_0^*X +\pi_1^*Y, \pi_0^*X' + \pi_1^*Y']|_S\\
&=[\pi_0^*X,\pi_0^*X']|_S + [\pi_1^*Y,\pi_1^*Y']|_S\\
&=(\pi_0^*[X,X'] + \pi_1^*[Y,Y'])|_S,
\end{split}\end{equation}
where we use the fact that sections pulled back from opposite
factors $M_0, M_1$ Courant commute in the product.  Applying the
projections to the final formula, we obtain
\[
p_0([Z,Z'])=[p_0(Z),p_0(Z')], \ \ \ \ \ p_1([Z,Z']) =
[p_1(Z),p_1(Z')],
\]
as required.
\end{proof}
We now describe the general form of a generalized complex brane when
it is supported on the whole manifold $M$; these are usually called
``space-filling branes''.  We first observe that the requirement
that $M$ be a generalized complex submanifold of itself places a
very strong constraint on $\JJ$.

\begin{prop}\label{selfsub}
$(M,\JJ)$ is a generalized complex submanifold of itself if and only
if there exists an integrable isotropic splitting $E=T\oplus T^*$ of
the Courant algebroid with respect to which $\JJ$ has the form:
\begin{equation}\label{upt}
\JJ = \begin{pmatrix}I&Q\\ & -I^*\end{pmatrix},
\end{equation}
where $I$ is a usual complex structure on the manifold and $\sigma =
P+iQ$, for $P=IQ$, is a holomorphic Poisson structure, i.e. it
satisfies $[\sigma,\sigma]=0$.
\end{prop}
\begin{proof}
Compatibility of the splitting with $\JJ$ forces $\JJ T = T$, which
holds iff $\JJ$ is upper triangular, and the orthogonality of $\JJ$
together with the fact $\JJ^2=-1$ guarantees that $I$ is an almost
complex structure and that $Q$ is a bivector of type $(2,0)+(0,2)$.
The $-i$-eigenbundle of $\JJ$ is then the direct sum of $T_{0,1}$
with the graph of $\sigma:T^*_{1,0}\lra T_{1,0}$.  This is closed
(involutive) for the Courant bracket if and only if $T_{0,1}$ is
integrable and $[\sigma,\sigma]= 0$, as required.
\end{proof}

In the splitting $E=T\oplus T^*$ for which $\JJ$ has the
form~\eqref{upt}, we see that $\tau_S = TM$, and further that $\ell
= T_{1,0}$, so that $\ell$-modules are precisely holomorphic bundles
with respect to the complex structure $I$.
\section{Multiple branes and holomorphic Poisson varieties}
Suppose that we have a Courant trivialization $s$ making $(M,\JJ)$ a
generalized complex submanifold of itself, so that $E=T\oplus T^*$
and $\JJ$ has the form~\eqref{upt}.  Now we investigate the
consequences of having a second trivialization $s'$ which is also
compatible with $\JJ$.  Let $F\in \Omega^2_{cl}(M,\RR)$ be the
2-form taking $s$ to $s'$. By Proposition~\ref{selfsub}, and the
fact that the Poisson structure $Q$ is independent of splitting, we
have
\begin{equation}\label{twobrane}
\begin{pmatrix}1&\\ -F& 1\end{pmatrix}\begin{pmatrix}I&Q\\ & -I^*\end{pmatrix}\begin{pmatrix}1&\\ F& 1\end{pmatrix}
=\begin{pmatrix}J&Q\\&-J^*\end{pmatrix},
\end{equation}
for a second complex structure $J$ such that $\sigma' = JQ + iQ$ is
holomorphic Poisson.  In particular we note the important fact that
a generalized complex structure may be expressed as a holomorphic
Poisson structure \emph{in several different ways}, and with respect
to different underlying complex structures, depending on the choice
of splitting.
 Equation~\eqref{twobrane} is equivalent to the conditions
\begin{align}\label{grpd}
\begin{cases} J-I = QF,\\
FJ + I^*F=0.
\end{cases}
\end{align}
Phrased as a single condition on $F$, we obtain the nonlinear
equation
\begin{equation}\label{nlin}
FI + I^*F + FQF = 0,
\end{equation}
which may be viewed as a deformation of the usual condition
$FI+I^*F=0$ that $F$ be of type $(1,1)$ with respect to the complex
structure.  Equation~\eqref{nlin} has been studied in~\cite{MR2055289},
where it was shown that it corresponds to a noncommutative version of the
$(1,1)$ condition via the Seiberg-Witten transform on tori.  We take
a different approach here, focusing rather on a groupoid
interpretation of the equivalent system~\eqref{grpd}.

The set of compatible global Courant trivializations forms a
groupoid; we may label each trivialization by the complex structure
it induces on the base, and we see from~\eqref{twobrane}
or~\eqref{grpd} that if $F_{IJ}$ takes $I$ to $J$ and $F_{JK}$ takes
$J$ to another trivialization $K$, then $F_{IJ}+F_{JK}$ takes $I$ to
$K$.
\begin{defn}\label{groupoid}
Fix a real manifold $M$ with real Poisson structure $Q$.  Let
$\mathcal{G}$ be the groupoid whose objects are holomorphic Poisson
structures $(I_i, \sigma_i)$ on $M$ with fixed imaginary part given
by $\mathrm{Im}(\sigma_i)=Q$, and whose morphisms $\Hom(i,j)$
consist of real closed 2-forms $F_{ij}\in \Omega^2_{cl}(M,\RR)$ such
that
\begin{align}\label{grpd2}
\begin{cases}
I_j-I_i = QF_{ij},\\
F_{ij}I_j + I_i^*F_{ij}=0.
\end{cases}
\end{align}
The composition of morphisms is then simply addition of 2-forms
$F_{ij}+F_{jk}$. In keeping with the interpretation of $F_{ij}$ as
differences between gerbe trivializations, we could define
$\Hom(i,j)$ to consist of unitary line bundles $L_{ij}$ with
curvature $F_{ij}$, such that composition of morphisms would
coincide with tensor product.
\end{defn}
Automorphisms of the Courant algebroid which fix $\JJ$ give rise to
automorphisms of the groupoid of trivializations defined above; we
describe these now.  Orthogonal automorphisms of the standard
Courant bracket on $T\oplus T^*$ consist of pairs $(\varphi,B)\in
\mathrm{Diff}(M)\times \Omega^2_{cl}(M,\RR)$, which act on $T\oplus
T^*$ via $X+\xi\mapsto \varphi_*X + (\varphi^{-1})^*\xi +
i_{\varphi_*X} B$. Since our generalized complex structure has the
form~\eqref{upt}, we may easily determine its automorphism group.
\begin{prop}\label{auto}
The automorphism group $\mathrm{Aut}(\JJ)$ of the generalized
complex structure~\eqref{upt} is the set of pairs
$(\varphi,B)\in\mathrm{Diff}(M)\times \Omega^2_{cl}(M,\RR)$ such
that
\begin{align*}
Q^\varphi &= Q\\
I^\varphi - I &= QB\\
BI^\varphi + I^*B &= 0,
\end{align*}
where $Q^\varphi = \varphi_*Q$ and $I^\varphi = \varphi_* I
\varphi_*^{-1}$.
\end{prop}
These automorphisms therefore act on the groupoid of global
generalized complex submanifolds~\eqref{grpd2}, sending
$(I_i,\sigma_i)\mapsto (I_i^\varphi, (\varphi^{-1})^*\sigma_i)$ and
$F_{ij}\mapsto (\varphi^{-1})^*F_{ij}+B$. Of course, we may wish to
interpret $B$ as the curvature of a unitary line bundle $U$, in
which case it would act on the groupoid line bundles $L_{ij}$ by
tensor product $L_{ij}\mapsto (\varphi^{-1})^*L_{ij}\otimes U$.

Instead of viewing $F_{ij}$ as the difference between two
generalized complex submanifolds of $(M,\JJ)$, we may interpret
Equation~\eqref{twobrane} as giving an isomorphism between two
different generalized complex structures on $T\oplus T^*$.  This
rephrasing leads immediately to the following.

\begin{prop}\label{diffint}
Let $(I_i,\sigma_i)$ and $(I_j,\sigma_j)$ be holomorphic Poisson
structures on $M$ with associated generalized complex structures
$\JJ_i,\JJ_j$ on $T\oplus T^*$ via~\eqref{upt}, let
$\mathrm{Im}(\sigma_i)=\mathrm{Im}(\sigma_j)=Q$, and let
$F_{ij}\in\Omega^2_{cl}(M,\RR)$ satisfy equation~\eqref{grpd2}. Then
the graph of $F_{ij}$ over the diagonal $\Delta\subset M\times M$
defines a generalized complex submanifold of $(M\times
M,\overline{\JJ}_i\times\JJ_j)$, yielding an isomorphism of
generalized complex manifolds
\begin{equation}\label{isogc2}
(M,\JJ_i)\stackrel{\cong}{\lra}(M, \JJ_j).
\end{equation}
\end{prop}
In view of Proposition~\ref{transbranes}, this result implies that a
morphism $F_{ij}$ from $(I_i,\sigma_i)$ to $(I_j,\sigma_j)$ induces
an equivalence between the categories of generalized holomorphic
bundles associated to $\JJ_i,\JJ_j$.  We now explain this
equivalence explicitly, and its significance for the holomorphic
Poisson structures involved.

\begin{prop}
Let $\JJ$ be of the form~\eqref{upt}, for $I$ a complex structure
and $\sigma = P+iQ$ a holomorphic Poisson structure.  Then a
generalized holomorphic bundle is precisely a holomorphic Poisson
module~\cite{MR1465521}, i.e. a holomorphic bundle $V$ with an additional
action of the structure sheaf on the sheaf of holomorphic sections,
denoted $\{f,s\}$, satisfying
\begin{align}
\{f, gs\} &= \{f,g\}s + g\{f,s\},\label{one}\\
\{\{f,g\},s\} &= \{f,\{g,s\}\} - \{g,\{f,s\}\}\label{two},
\end{align}
where $f,g\in \mathcal{O}$, $s\in\mathcal{O}(V)$, and $\{f,g\}$
denotes the Poisson bracket induced by $\sigma$.
\end{prop}
\begin{proof}
Let $L=\ker (\JJ + i1)$, so that for $\JJ$ as in~\eqref{upt}, we
have
\begin{equation}\label{eldc}
L = T_{0,1}\oplus \Gamma_{\sigma},
\end{equation}
where $\Gamma_{\sigma}=\{\xi + \sigma(\xi)\ :\ \xi\in T^*_{1,0}\}$.
Let $\delbar_V: C^\infty(V)\lra C^\infty(L^*\otimes V)$ be a
generalized holomorphic structure. Decomposing using~\eqref{eldc}
and identifying $\Gamma_\sigma = T^*_{1,0}$, we write $\delbar_V
=\delbar' +\delbar''$, where $\delbar':C^\infty(V)\lra
C^\infty(T^*_{0,1}\otimes V)$ is a usual holomorphic structure and
$\delbar'':C^\infty(V)\lra C^\infty(T_{1,0}\otimes V)$ satisfies,
for $f\in C^\infty(M,\CC)$ and $s\in C^\infty(V)$,
\[
\delbar''(fs) = f\delbar''s + Z_f\otimes s,
\]
where $Z_f=\sigma(df)$ is the Hamiltonian vector field of $f$.  This
is equivalent to condition~\eqref{one}. Furthermore $\delbar_V^2=0$
implies that $\delbar''$ is holomorphic and defines a Poisson module
structure via
\[
\{f, s\} = \delbar''_{\del f} s,
\]
since $(\delbar'')^2=0$ implies
$\delbar''_{\del\{f,g\}}=[\delbar''_{\del f},\delbar''_{\del g}]$,
which is equivalent to~\eqref{two}, as required.
\end{proof}
Letting $L_k=\ker(\JJ_k + i1)$, we see from
Equation~\eqref{twobrane} that $\exp(F_{ij})$ takes $L_i$ to $L_j$.
Hence the map on generalized holomorphic bundles induced by the
isomorphism~\eqref{isogc2} may be described as composition with
$\exp(F_{ij})$
\[
\delbar_i\mapsto e^{F_{ij}}\circ\delbar_i.
\]
This map may be made more explicit in terms of the associated
generalized connections.  Choose a Hermitian structure on the
$\JJ_i$-holomorphic bundle (i.e. $\sigma_i$-Poisson module), and let
$D=\nabla + \chi$ be the extension of $\delbar_i$ as in
Proposition~\ref{herm}.  Then $F_{ij}$ acts on $D$ via
\begin{equation}\label{expmap}
\begin{cases}
D\mapsto D' = \nabla' +  \chi,\\
\nabla' = \nabla + F_{ij}(\chi).
\end{cases}
\end{equation}
which then defines a $\sigma_j$-Poisson module. It is important to
note that the $\sigma_i$-Poisson module, which is $I_i$-holomorphic,
inherits via~\eqref{expmap} a $I_j$-holomorphic structure, without
the presence of any holomorphic map between $(M,I_j)$ and $(M,I_i)$.

Given this result, it is natural to ask how restrictive the
condition of admitting a Poisson module structure actually is.  The
following is a simple result describing the complete obstruction to
the existence of a Poisson module structure on a holomorphic line
bundle.
\begin{prop}\label{obstruction}
Let $M$ be a holomorphic Poisson manifold, and let $V$ be a
holomorphic line bundle on $M$. Then the Atiyah class of $V$,
$\alpha\in H^1(T^*_{1,0})$, combines with the Poisson structure
$\sigma\in H^0(\wedge^2 T_{1,0})$ to give the class $\sigma\alpha\in
H^1(T_{1,0})$. If $\sigma\alpha=0$, then there is a well-defined
secondary characteristic class $f_\alpha\in H^2_\sigma(M)$ in
Poisson cohomology.  $V$ admits a Poisson module structure if and
only if both classes $ \{\sigma\alpha,f_\alpha\}$ vanish.  The space
of Poisson module structures is affine, modeled on $H^1_\sigma(M)$.
\end{prop}
\begin{proof}
A Poisson module structure on $V$ is a holomorphic differential
operator $\del:\mathcal{O}(V)\lra \mathcal{O}(T_{1,0}\otimes V)$
satisfying $\del(fs) =f\del s + Z_f\otimes s$, for $f\in\mathcal{O}$
and $s\in \mathcal{O}(V)$, where $Z_f$ is the $\sigma$-Hamiltonian
vector field of $f$, and such that the curvature vanishes.  Let
$\{U_i\}$ be an open cover of $M$ and let $\{s_i\in
\mathcal{O}(U_i,V)\}$ be a local trivialization of $V$ such that
$s_i = g_{ij} s_j$ for holomorphic transition functions $g_{ij}$;
then
\begin{equation}\label{locconn}
\del s_i = X_i\otimes s_i,
\end{equation}
where $X_i$ are holomorphic Poisson vector fields (since
$\del_{d\{f,g\}} = [\del_{df},\del_{dg}]$) such that
\begin{equation}\label{pc}
X_i-X_j = Z_{\log g_{ij}}.
\end{equation}
The Hamiltonian vector fields $Z_{\log g_{ij}} = \sigma(d\log
g_{ij})$ are a \v{C}ech representative for the image of the Atiyah
class under $\sigma$. Therefore, equation~\eqref{pc} holds if and
only if $\sigma\alpha=0\in H^1(T_{1,0})$.   If $\sigma\alpha=0$,
then we may solve~\eqref{pc} for some holomorphic vector fields
$\tilde X_i$. We can modify these by a global holomorphic vector
field so that they are each Poisson if and only if the global
bivector field $f_\sigma$ defined by $f_\sigma|_{U_i} = [\tilde
X_i,\sigma]$ vanishes in Poisson cohomology, i.e. $f_\sigma =
[Y,\sigma]$ for $Y\in H^0(T_{1,0})$, in which case $X_i =\tilde X_i
- Y|_{U_i}$ defines a Poisson module structure as required.

Given any holomorphic Poisson vector field $Z\in H^1_\sigma(M)$ and
Poisson module structure $\del$, the sum $\del+Z$ defines a new
Poisson module structure.  Conversely, two Poisson module structures
$\del',\del$ must satisfy $\del'-\del\in H^1_\sigma(M)$, as claimed.
\end{proof}

It is remarked in~\cite{MR1465521} that the canonical line bundle $K$
always admits a natural Poisson module structure for any holomorphic
Poisson structure $\sigma$ via the action, for $f\in\mathcal{O}$ and
$\rho\in\mathcal{O}(K)$,
\[
\{f,\rho\} = L_{Z_f}\rho.
\]
Based on these considerations, we obtain the following example.
\begin{example}
Let $M=\CC P^2$, equipped with a holomorphic Poisson structure
$\sigma\in H^0(\mathcal{O}(3))$.  Note that $H^1(T_{1,0})=0$. Then
$K=\mathcal{O}(-3)$ is canonically a Poisson module, and since
$\mathcal{O}(1)^{-3}=K$, we see that the obstruction $f_\sigma$ from
Proposition~\ref{obstruction} must vanish for $\mathcal{O}(1)$ as
well (note that $\dim H^2_\sigma(\CC P^2)=2$, so the obstruction
space is nonzero). Hence all holomorphic line bundles
$\mathcal{O}(k)$ admit Poisson module structures.  If $\sigma$ is
generic, these Poisson module structures are unique, since
$H^1_\sigma(M)=0$, due to the fact that only the zero holomorphic
vector field on $\CC P^2$ stabilizes a smooth cubic curve.
\end{example}

We conclude this section with a simple example of a generalized
complex manifold admitting multiple trivializations with
\emph{non-biholomorphic} induced complex structures.

\begin{prop}
Let $E_0= E\times \CC$, the trivial line bundle over an elliptic
curve $E$, and let $E_c$, for $c\in\RR$, be the alternative
holomorphic structure on $E\times\CC$ obtained by endowing the
bundle $E\times\CC$ with the holomorphic structure associated to the
point $ic\in H^1(E,\mathcal{O})=\CC$.  Then $E_0$ and $E_c$ are
diffeomorphic, non-biholomorphic complex manifolds.  They are
equipped with canonical holomorphic Poisson structures $\sigma_0,
\sigma_c$ vanishing to first order on the zero section, and
furthermore $(E_0,\sigma_0)$ and $(E_c,\sigma_c)$ are isomorphic as
generalized complex manifolds $\forall c\in\RR$ (and hence have
equivalent categories of Poisson modules).
\end{prop}
\begin{proof}
Represent $E$ as $\CC^*/\{z\mapsto \lambda z\}$ and let $w$ be the
linear coordinate on the fiber of $E\times \CC$.  Then the
holomorphic structure $E_c$ is given by the complex volume form
\[
\Omega_c = \frac{dz}{z}\wedge\left(dw + icw\frac{d\bar z}{\bar
z}\right),
\]
and the holomorphic Poisson structure $\sigma_c$ is given by
\[
\sigma_c = \left(z\frac{\del}{\del z} + ic \bar w\frac{\del}{\del
\bar w}\right)\wedge w\frac{\del}{\del w}
\]
The pure spinor corresponding to the generalized complex structure
$(E_c,\sigma_c)$ is
\[
\rho_c = e^{\sigma_c}\Omega_c.
\]
Now let $F_c = ic\tfrac{dz\wedge d\bar z}{z\bar z}$ be a real
multiple of the volume form on $E$ (which may be viewed as a
curvature when $c\in 2\pi\ZZ$).  Then we verify that
\begin{align*}
e^{F_c}e^{\sigma_0}\Omega_0 &= w +
\tfrac{dz}{z}\wedge\left(dw+icw\tfrac{d\bar z}{\bar z}\right)\\
&=e^{\sigma_c}\Omega_c,
\end{align*}
showing that $(E_0,\sigma_0)$ and $(E_c,\sigma_c)$ are isomorphic as
generalized complex manifolds.
\end{proof}

\section{Relation to generalized K\"ahler geometry}

A generalized K\"ahler structure is a pair $(\JJ_A,\JJ_B)$ of
commuting generalized complex structures such that
\[
 G(\cdot,\cdot)=\IP{\JJ_A\cdot,\JJ_B\cdot}
\]
is a generalized Riemannian metric.

In~\cite{Gualtieri:rp} it is shown that the integrability of the pair $(\JJ_A,\JJ_B)$ is equivalent to the fact that the induced decomposition of the definite subspaces $C_\pm$ given by
\[
 C_\pm\otimes\CC = L_\pm\oplus \overline{L_\pm},
\]
where $L_\pm = \ker(\JJ_A-i1)\cap\ker(\JJ_B\mp i1)$, satisfies the
condition that $L_\pm$ are each involutive. Using the
generalized Bismut connection $D$ introduced in Theorem~\ref{lclc}, we
provide the following equivalent description of generalized K\"ahler
geometry.
\begin{theorem}
Let $G$ be a generalized metric and let $\JJ$ be a $G$-orthogonal
almost generalized complex structure.  Then $(\JJ,G)$ defines a
generalized K\"ahler structure if and only if $D\JJ=0$ and the
torsion $T_D\in C^\infty(\wedge^3E)$ is of type $(2,1)+(1,2)$ with
respect to $\JJ$.
\end{theorem}
\begin{proof}
We leave the forward direction to the reader.  We show that if
$D\JJ=0$ (where, as usual, $(D_x\JJ)y = D_x(\JJ y) - \JJ(D_xy)$) and
the torsion is as above, then $\JJ$ is integrable as a generalized
complex structure. Note that under these assumptions, the
complementary generalized complex structure $\JJ' = G\JJ$ would also
be covariant constant, and be compatible with the torsion as well,
by Proposition~\ref{torscalc}. Therefore by the following argument
$\JJ'$ is also integrable, and we obtain the result.

We compute the Nijenhuis tensor of $\JJ$, for $x,y,z\in C^\infty(E)$
(in the following, $[\cdot,\cdot]$ refers to the skew-symmetrized
Courant bracket):
\begin{align*}
\IP{N_{\JJ}(x,y),z}&= \IP{[\JJ x,\JJ y] - \JJ[\JJ x,y] - \JJ[x,\JJ y] -[x,y],z}\\
&=\IP{D_z (\JJ x),\JJ y} -\IP{D_z(\JJ y),\JJ x}+\IP{D_{\JJ z} (\JJ
x),y} - \IP{D_{\JJ z} y, \JJ x}\\
&+\IP{D_{\JJ z} x,\JJ y} -\IP{D_{\JJ z}(\JJ y),x} -\IP{D_z x,y} +
\IP{D_z y,x}\\
&-T_D(\JJ x,\JJ y, z) - T_D(\JJ x, y, \JJ z) - T_D(x,\JJ y, \JJ z) -
T(x,y,z).
\end{align*}
The first eight terms cancel since $D_x(\JJ y) = \JJ D_x y$, and the
last four terms cancel since $T_D$ is of type $(2,1)+(1,2)$.
Therefore $\JJ$ is integrable, as claimed.
\end{proof}

We now explain that a solution to the system~\eqref{grpd}, if
positive in a certain sense, gives rise to a generalized K\"ahler
structure.  When the Poisson structure $Q$ vanishes, this result
specializes to the fact that a positive holomorphic line bundle with
Hermitian structure defines a K\"ahler structure.

\begin{defn}\label{positive}
Let $(I,J,Q,F)$ be a solution to the system~\eqref{grpd}, i.e. it
defines two global Courant trivializations compatible with a
generalized complex structure, separated by the 2-form $F$.  Then
\[
g = -\tfrac{1}{2}F(I+J)
\]
is a symmetric tensor, and if it is positive-definite, we say that
$F$ is \emph{positive}.
\end{defn}
If $F$ is positive, then $(g,I)$, $(g,J)$ are both Hermitian
structures.  Let $\omega_I=gI$, $\omega_J=gJ$ be their associated
2-forms.  Then we have the following.
\begin{theorem}\label{branekahler}
Let $(I,J,Q,F)$ be as above, and let $F$ be positive.  Then the pair
\begin{equation}\label{form}\begin{split}
\JJ_{B}=\frac{1}{2}\left(\begin{matrix}1&\\b&1\end{matrix}\right)
\left(\begin{matrix}J + I & -(\omega_J^{-1}-\omega_I^{-1}) \\
\omega_J-\omega_I&-(J^*+ I^*)\end{matrix}\right)
\left(\begin{matrix}1&\\-b&1\end{matrix}\right),\\
\JJ_{A}=\frac{1}{2}\left(\begin{matrix}1&\\b&1\end{matrix}\right)
\left(\begin{matrix}J - I & -(\omega_J^{-1}+\omega_I^{-1}) \\
\omega_J+\omega_I&-(J^*- I^*)\end{matrix}\right)
\left(\begin{matrix}1&\\-b&1\end{matrix}\right),\\
\end{split}\end{equation}
defines a generalized K\"ahler structure on the standard Courant
algebroid $(T\oplus T^*,[\cdot,\cdot]_0)$, for $b\in\Omega^2(M,\RR)$
given by
\[
b = -\tfrac{1}{2}F(J-I).
\]
\end{theorem}
\begin{proof}
It is easily verified that $\JJ_A^2=\JJ_B^2=-1$ and that
$[\JJ_A,\JJ_B]=0$.  To show integrability, we first observe that
$\JJ_A$ has the form of a pure symplectic structure; indeed, with
the definitions above,
\[
\JJ_A = \begin{pmatrix}&-F^{-1}\\F&\end{pmatrix}.
\]
We see therefore that $\JJ_A$ is integrable since $dF=0$.

The structure $\JJ_B$ is also integrable, as follows.  Let $L_B =
\ker (\JJ_B-i)$ and let $L_B = L_+ \oplus L_-$ be its decomposition
into $\pm i$ eigenspaces for $\JJ_A$.  Then
\begin{align*}
L_+ &= \{X + (b+ g)X\ :\ X\in T^{1,0}_J\}\\
L_- &= \{X + (b- g)X\ :\ X\in T^{1,0}_I\}
\end{align*}
It follows from the definitions of $b,g$ that $b+g = -FJ$ whereas
$b-g = FI$.  As a result we have
\begin{align*}
L_+ = \{X -i FX \ :\ X\in T^{1,0}_J\}\\
L_- = \{X +i FX \ :\ X\in T^{1,0}_I\}
\end{align*}
which are integrable precisely when $i_Xi_Y dF=0$ for all $X,Y$ in
$T^{1,0}_I$ or $T^{1,0}_J$. Of course this holds since $F$ is
closed.
\end{proof}

We note that the converse of this argument also holds; using the
result from~\cite{Gualtieri:rp} that any generalized K\"ahler structure has
the form~\eqref{form}, we may show that any generalized K\"ahler
structure  $(\JJ_A,\JJ_B)$ with the property that $\JJ_A$ is
symplectic gives rise to a solution to the system~\eqref{grpd}. More
explicitly, given the bi-Hermitian data $(g,I,J)$ we determine $F$
via
\[
F = -2g(I+J)^{-1},
\]
where $(I+J)$ is invertible by the assumption on $\JJ_A$, and the
Poisson structure $Q$ is given by
\[
Q = (J-I)F^{-1} = \tfrac{1}{2}[I,J]g^{-1}.
\]
This is consistent with Hitchin's general
observation~\cite{MR2217300} that $[I,J]g^{-1}$ defines a
holomorphic Poisson structure for both $I$ and $J$, for any
generalized K\"ahler structure.

In fact, the interpretation of $F_{ij}$ in Proposition~\ref{diffint}
as defining a morphism between holomorphic Poisson structures allows
us to view the generalized K\"ahler structure as a morphism between
the holomorphic Poisson structures $(I,\sigma_I)$, $(J,\sigma_J)$.
This point of view is related to the approach in~\cite{MR2276362} to
defining a generalized K\"ahler potential, and may help to resolve
the problems encountered there at non-regular points.

Given the equivalence between certain generalized K\"ahler
structures and configurations of generalized complex submanifolds
shown in this section, we may apply it to produce new examples of
generalized K\"ahler structures, or indeed of configurations of
branes. We do this in the following section.
\section{Construction of generalized K\"ahler metrics}\label{construct}

Given a generalized complex submanifold, it is natural to construct
more by deformation; this is a familiar construction in symplectic
geometry, where new Lagrangian submanifolds may be produced by
applying Hamiltonian or symplectic diffeomorphisms.  Therefore we
would like to deform a given generalized complex submanifold by an
automorphism of the underlying geometry, as described in
Proposition~\ref{auto}.  If the automorphism used is positive in the
sense of Definition~\ref{positive}, then we will have constructed a
generalized K\"ahler structure, by Theorem~\ref{branekahler}.  This
construction is inspired by a construction of Joyce contained
in~\cite{MR1702248}, and its generalization by~\cite{MR2217300} to
the construction of generalized K\"ahler structures on Del Pezzo
surfaces.

To reiterate, the goal of the construction is as follows: given a
holomorphic Poisson structure $(I,\sigma_I)$ on $M$, with real and
imaginary parts $\sigma_I = P + iQ$, find a second complex structure
$J$ and a 2-form $F$ solving the system~\eqref{grpd}, i.e.
\begin{align}\label{grpd3}
\begin{cases} J-I = QF,\\
FJ + I^*F=0.
\end{cases}
\end{align}
We are particularly interested in the case where
$g=-\tfrac{1}{2}F(I+J)$ is positive-definite, as this then defines a
generalized K\"ahler structure, however the construction does not
depend on it.

In this construction, the complex structure $J$ will be obtained
from $I$ by flowing along a vector field; as a result, $J$ will be
biholomorphic to $I$.  Also, we shall describe the construction in
the case that $F$ is the curvature of a unitary connection, although
it will be clear that integrality of the form $F$ is not required.
\begin{itemize}
\item[1.]
We begin with a Hermitian complex line bundle $L$ over a compact
complex manifold $M$; the 2-form $F$ solving~\eqref{grpd3} will be
chosen from the cohomology class $c_1(L)$. We first assume that $L$
admits a holomorphic structure $\delbar_0$ with respect to the
``initial'' complex structure $I=I_0$.  The associated Chern
connection will be called $\nabla_0$, and its curvature denoted
$F_0$. Recall that $\nabla_0$ is the unique Hermitian connection on
$L$ such that $\nabla_0^{0,1} = \delbar_0$.

\item[2.]
We then assume that $L$ admits the structure of a holomorphic
Poisson module with respect to a holomorphic Poisson structure
$\sigma_I$ on $M$, which by Proposition~\ref{obstruction} occurs if
and only if $[\sigma_I F_0] \in H^1_I(T_{1,0})$ vanishes and the
secondary characteristic class in $H^2_{\sigma_I}(M)$ also vanishes.
By Proposition~\ref{herm}, we construct the Hermitian generalized
connection $D$ associated to this generalized holomorphic structure,
and decompose it according to the splitting $T\oplus T^*$:
\[
D = \nabla_0 + iX,
\]
where $X$ is a real $Q$-Poisson vector field such that $\delbar
X^{1,0} = \sigma_I F_0$, giving rise to the real equations
\begin{equation}\label{key}
\begin{cases}
L_X Q &= 0,\\
L_X I_0 &= QF_0.
\end{cases}
\end{equation}

\item[3.] Let $\varphi_t$ be the time-$t$ flow of the vector field
$X$.  Then we may transport $F_0$ by the flow, yielding the
cohomologous family of 2-forms $F_t = \varphi_{-t}^*F_0$, which
satisfies
\[
\dot F_t = L_X F_t = di_X F_t.
\]
We may also transport $I_0$ by the flow, obtaining a family $I_t =
I_0^{\varphi_t}$ satisfying
\[
\dot I_t = L_X I_t = QF_t,
\]
by Equation~\eqref{key}.  Note that $F_t$ is type $(1,1)$ with
respect to $I_t$. Also note that $F_t$ is the curvature of the
family of connections
\[
\nabla_t = \nabla_0 + \int_0^t i_X F_s ds,
\]
which are therefore the Chern connections associated to a family of
holomorphic structures $\delbar_t$ on $L$, each holomorphic with
respect to $I_t$.

\item[4.] We then compute the difference
\begin{align}
I_t - I_0 &= \int_0^t Q F_s ds\notag\\
&= tQ \tfrac{1}{t}\int_0^t F_s ds\label{solve}\\
&= tQ \overline F_t,\notag
\end{align}
where $\overline F_t$ is the curvature of the average Chern
connection on $L$:
\[
\overline\nabla_t = \tfrac{1}{t}\int_0^t \nabla_s ds.
\]
Setting $t=1$ we obtain a solution to the first part
of~\eqref{grpd3}:
\[
I_1 - I_0 = Q\overline F_1.
\]
\item[5.]
Observe that the second part of~\eqref{grpd3} is automatically
satisfied: from~\eqref{solve} we have $I_t - I_0 = QG_t$, where
\[
G_t = \int_0^t F_s ds.
\]
For $t=0$, the expression
\begin{equation}\label{van}
G_t I_t + I_0^*G_t
\end{equation}
vanishes, since $G_0=0$.  Taking the time derivative, we obtain
\begin{align*}
\dot G_t I_t + G_t \dot I_t + I_0^*\dot G_t &= F_tI_t +G_tQF_t +
I_0^*F_t\\
&=-(I_t^* - I_0^*)F_t + G_t QF_t\\
&=-G_tQF_t + G_tQF_t = 0.
\end{align*}
Therefore~\eqref{van} vanishes for all $t$; since
$\overline{F_t}=t^{-1}G_t$, we obtain the result.
\item[6.] \emph{Positivity:} If $F_0$ is positive, i.e. if the
original line bundle $L$ is positive, then $\overline F_t$ is
positive for sufficiently small $t$.  By Equation~\eqref{solve},
this gives a solution to the system~\eqref{grpd3} for the Poisson
structure $t\sigma_I$ replacing $\sigma_I$.
\end{itemize}

We summarize the main result of this construction in the following.
\begin{theorem}
Let $L$ be a positive holomorphic line bundle with Poisson module
structure over a compact complex manifold with holomorphic Poisson
structure $\sigma$. Let $(g_0,I_0)$ be the original K\"ahler
structure it determines.  Then the choice of Hermitian structure on
$L$ determines a canonical family of generalized K\"ahler structures
$\{(g_t,I_t,I_0)\ :\ -\epsilon<t<\epsilon\}$ such that the complex
structure $I_t$ coincides with $I_0$ only along the vanishing locus
of $\sigma$ for $t\neq 0$.
\end{theorem}

\begin{example}
One case where the existence of a positive Poisson module is
guaranteed is in the case of a Fano Poisson manifold, since the
anticanonical bundle, which always admits a Poisson module
structure, is positive.  This extends the result of~\cite{MR2371181},
who showed that all smooth Fano surfaces (the Del Pezzo surfaces)
admit the families of generalized K\"ahler structures described
here.
\end{example}

We remark finally upon the relation of our construction to Hitchin's
result for Del Pezzo surfaces~\cite{MR2371181}.  To obtain the family of
generalized K\"ahler structures, he used a flow generated by a
Poisson vector field $X$ which he expressed as the Hamiltonian
vector field of $\log |s|^2$, for $s$ a holomorphic section of the
anticanonical bundle vanishing at the zero locus of the Poisson
structure. From our point of view, he was making use of the
generalized Poincar\'e-Lelong formula of Proposition~\eqref{PL},
since in the 2-dimensional case there is always a non-trivial
generalized holomorphic section of the anti-canonical bundle of a
Poisson surface, namely the Poisson structure itself.  However, in
higher dimension, there is a dearth of global generalized
holomorphic sections; indeed by Proposition~\eqref{PL}, such a
section (if generically nonzero) must vanish only along the zero
locus of $\sigma$, which has codimension greater than one in
general.
\section{Relation to non-commutative algebraic geometry}

Since the observation in~\cite{Gualtieri:rp} that the deformation space of a
complex manifold as a generalized complex manifold includes the
``noncommutative'' directions in $H^0(\wedge^2T_{1,0})$, it was
hoped that there might be a more precise relationship between
generalized complex structures and noncommutativity.  The presence
of an underlying Poisson structure, for example, lends credence to
this idea. In the realm of generalized K\"ahler 4-manifolds, we have
even more evidence in this direction, since, as observed originally
in~\cite{MR1702248}, the locus where the bi-Hermitian complex
structures $(I_+,I_-)$ coincide is an anti-canonical divisor for
both structures.

If smooth, each connected component of this coincidence locus is an
elliptic curve $C$, and we may view it as embedded in two different
complex manifolds $X_\pm = (M,I_\pm)$.
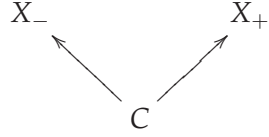
\begin{equation*}
\xymatrix{X_-& & X_+\\ &C\ar[lu]\ar[ru] & }
\end{equation*}
In the examples constructed in Section~\ref{construct} and
by~\cite{MR2371181}, $X_\pm$ have natural holomorphic line bundles
sitting over them, which we called $L_0,L_1$.  Pulling them back to
$C$, we obtain holomorphic line bundles
$\mathcal{L}_0,\mathcal{L}_1$ over $C$.  Furthermore, since the flow
satisfied $L_XI_t = QF_t$, we know that the flow restricts to a
holomorphic flow on $C$, the vanishing locus of $Q$.   As a result,
$\mathcal{L}_0,\mathcal{L}_1$ are related by an automorphism of $C$.
This data $(C,\mathcal{L}_0,\mathcal{L}_1)$ is precisely what is
used in the approach of~\cite{Bondal} to the classification of
$\ZZ$-algebras describing noncommutative projective surfaces.

In fact, in our construction we produce an example of an
automorphism $(\varphi,F)=(\varphi_1,\overline{F}_1)$ in the sense
of Proposition~\ref{auto}. Therefore we may apply it successively,
producing an infinite family of generalized complex submanifolds
with induced complex structures $\{I_k = I_0^{\varphi_1^k}\}$, each
$I_k$ separated from $I_0$ by the line bundle $L^k$ with connection
$\overline\nabla_k$, and all coinciding on the vanishing locus $C$
of $Q$. As a result we obtain an infinite family of embeddings
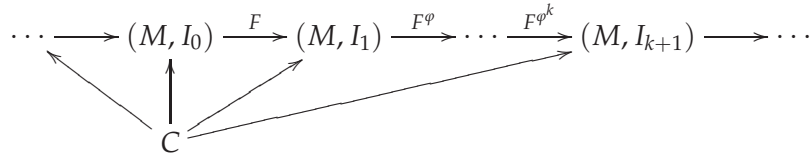
\begin{equation*}
\xymatrix{\cdots\ar[r]&(M,I_0)\ar[r]^-{F} &
(M,I_1)\ar[r]^-{F^\varphi}&\cdots\ar[r]^-{F^{\varphi^k}}&
(M,I_{k+1})\ar[r]&\cdots\\ & C\ar[ul]\ar[u]\ar[ur]\ar[urrr]}
\end{equation*}
Where the arrows on the top row indicate morphisms in the sense of
the groupoid of Definition~\ref{groupoid}.  This may provide an
alternative interpretation of Van den Bergh's construction of the
twisted homogeneous coordinate ring (see~\cite{VDB}): let
$\mathcal{L} = L_0|_C$, and let $\LL^\varphi=(\varphi^{-1})^*\LL$.
Then define the vector spaces
\[
\Hom(i,j) =
H^0(C,\LL^{\varphi^i}\otimes\LL^{\varphi^{i+1}}\otimes\cdots\otimes\LL^{\varphi^{j-1}})
\]
and define a $\ZZ^{>0}$-graded algebra structure on
\begin{equation}\label{vdb}
A^\bullet = \bigoplus_{k\geq 0} \Hom(0,k),
\end{equation}
via the multiplication, for $a\in A^p$ and $b\in A^q$:
\[
a\cdot b = a\otimes b^{\varphi^p},
\]
where we use the natural map $b\mapsto b^{\varphi^p}$ taking
$\Hom(0,q)\lra \Hom(p,p+q)$, and the tensor product is viewed as a
composition of morphisms.

Of course this is nothing but a recasting of the Van den Bergh
construction; there is a sense in which it captures only certain
morphisms between the generalized complex submanifolds, namely those
which are visible upon restriction to $C$.  Though rare, there are
sometimes generalized holomorphic sections of the bundles $L^k$
supported over all of $M$.  In some sense, these sections must be
included in the morphism spaces as well.

For instance, performing our construction for $L=\OO(1)$ over $\CC
P^2$, equipped with a holomorphic Poisson structure $\sigma\in
H^0(\CC P^2, \OO(3))$ with smooth zero locus $\iota:C\hookrightarrow
\CC P^2$, the graded algebra~\eqref{vdb} has linear growth instead
of the quadratic growth needed to capture a full non-commutative
deformation of the coordinate ring of $\CC P^2$ (these are the
Sklyanin algebras, classified by~\cite{ATV}).  It fails to include
an additional generator in degree $3$, as can be seen from the fact
that the restriction map $H^0(\CC P^2,\OO(3))\lra
H^0(C,\iota^*\OO(3))$ has 1-dimensional kernel. However it is
important to note that neither $\OO(1)$ nor $\OO(2)$ has generalized
holomorphic sections over $\CC P^2$, while $\OO(3)$ has a
1-dimensional space of them.  We end with this vague indication that
the morphisms supported on $C$ should be combined with those
supported on the whole holomorphic Poisson manifold.

\bibliographystyle{utphysmod}
\bibliography{GKG}

\end{document}